\documentclass[10pt]{amsart}

\usepackage[colorlinks]{hyperref}
\usepackage{color,graphicx,shortvrb}
\usepackage[latin 1]{inputenc}

\newtheorem{theorem}{Theorem}[section]

\newtheorem{corollary}[theorem]{Corollary}

\newtheorem{lemma}[theorem]{Lemma}

\newtheorem{proposition}[theorem]{Proposition}
\newtheorem{remark}[theorem]{Remark}

\usepackage[active]{srcltx} 

\usepackage{stackrel}
\usepackage{amssymb}

\def\J#1#2#3{ \left\{ #1,#2,#3 \right\} }

\def\NN{{\mathbb{N}}}
\def\11{\textbf{$1$}}

\usepackage{enumerate}
\usepackage{amsmath}
\usepackage{amssymb}
\usepackage{mathabx}

\begin{document}

\title[One-parameter groups of orthogonality preservers on C$^*$-algebras]{One-parameter groups of orthogonality preservers on C$^*$-algebras}

\author[J.J. Garc{\' e}s]{Jorge J. Garc{\' e}s}
\email{j.garces@upm.es}
\address{Departamento de Matem{\' a}tica Aplicada a la Ingenier{\' i}a Industrial, ETSIDI, Universidad Polit{\' e}cnica de Madrid, Madrid, Spain.}

\author[A.M. Peralta]{Antonio M. Peralta}
\email{aperalta@ugr.es}
\address{Departamento de An{\'a}lisis Matem{\'a}tico, Facultad de
Ciencias, Universidad de Granada, 18071 Granada, Spain.}

\dedicatory{This paper is dedicated to the memory of Professor Ottmar Loos, a pioneer of Jordan theory.}

\subjclass[2010]{Primary 46L05; 46L70; 47B48, Secondary 46K70; 46L40; 47B47; 47B49; 46B04; 17A40; 17C65.}

\keywords{C$^*$-algebra, JB$^*$-algebra, orthogonality preserver, one-parameter semigroups, one-parameter semigroups of orthogonality preserving operators}

\date{}

\maketitle

\begin{abstract}
 We establish a more precise description of those surjective or bijective continuous linear operators preserving orthogonality between C$^*$-algebras. The new description is applied to determine all uniformly continuous one-parameter semigroups of orthogonality preserving operators on an arbitrary C$^*$-algebra. We prove that given a family $\{T_t: t\in \mathbb{R}_0^{+}\}$ of orthogonality preserving bounded linear bijections on a general C$^*$-algebra $A$ with $T_0=Id$, if for each $t\geq 0,$ we set $h_t = T_t^{**} (1)$ and we write $r_t$ for the range partial isometry of $h_t$ in $A^{**},$ and $S_t$ stands for the triple isomorphism on $A$ associated with $T_t$ satisfying $h_t^* S_t(x)$  $= S_t(x^*)^* h_t$, $h_t S_t(x^*)^* =$ $ S_t(x) h_t^*$, $h_t r_t^* S_t(x) =$  $S_t(x) r_t^* h_t$, and $T_t(x) = h_t r_t^* S_t(x) = S_t(x) r_t^* h_t, \hbox{ for all } x\in A,$ the following statements are equivalent:\begin{enumerate}[$(a)$]\item $\{T_t: t\in \mathbb{R}_0^{+}\}$ is a uniformly continuous  one-parameter semigroup of orthogonality preserving operators on $A$;
\item $\{S_t: t\in \mathbb{R}_0^{+}\}$ is a uniformly continuous one-parameter semigroup of surjective linear isometries (i.e. triple isomorphisms) on $A$ (and hence there exists a triple derivation $\delta$ on $A$ such that $S_t = e^{t \delta}$ for all $t\in \mathbb{R}$), the mapping $t\mapsto h_t $ is continuous at zero, and the identity $ h_{t+s} =  h_t r_t^* S_t^{**} (h_s),$ holds for all $s,t\in \mathbb{R}.$\\
\end{enumerate}
\end{abstract}

\section{Introduction}
Let us address a typical problem on ``preservers''. Suppose $\mathcal{R}$ is a relation between elements in any C$^*$-algebra $A$ which is defined by the C$^*$- structure. A (linear) mapping $T$ from $A$ to a C$^*$-algebra $B$ preserves the relation $\mathcal{R}$ if $a\mathcal{R} b$ in $A$ implies $T(a)\mathcal{R} T(b)$ in $B$. We say that $T$ is bi-$R$-preserving if $a\mathcal{R} b$ in $A$ is equivalent to $T(a)\mathcal{R} T(b)$ in $B$. A very special protagonism is played by the relations ``being orthogonal'' and ``being disjoint''. Elements $a,b$ in a C$^*$-algebra $A$ are \emph{orthogonal} (denoted by $a\perp b$) if $a b^* = b^* a =0$. We say that $a$ and $b$ are \emph{disjoint} or have \emph{zero product} if $a b =0$. These particular relations give rise to the notions of linear (bi-)orthogonality preserving and linear (bi-)separating maps between C$^*$-algebras. These classes of maps have been intensively studied in recent years.\smallskip

In the setting of commutative C$^*$-algebras, elements are orthogonal if and only if they are disjoint. W. Arendt discovered in \cite{Arendt} that every bounded linear operator preserving orthogonality between $C(K)$-spaces must be a weighted composition operator. K. Jarosz determined all (non-necessarily continuous) linear maps preserving orthogonality between $C(K)$-spaces in \cite{Jar90}. One of the results in the just quoted reference shows that every linear bijection preserving orthogonality between $C(K)$-spaces is automatically continuous and biorthogonality preserving. Jarosz' result gave rise to a fruitful research line on automatic continuity of bi-separating or biorthogonality preserving linear maps. For example, every
zero-products preserving linear bijection from a properly infinite von Neumann algebra into a unital ring is automatically continuous \cite[Theorem 4.2]{CheKeLeeWong03}. Each linear bijection which preserves zero-products in both directions between standard subalgebras of $B(X)$ spaces (where $X$ is a general Banach space) is automatically continuous and a multiple of an algebra isomorphism \cite[Theorem 1]{ArJar03}. J. Araujo and K. Jarosz conjectured, in the just quoted reference, that every linear bijection between C$^*$-algebras preserving zero-products in both directions must be automatically continuous. We can find several attempts to prove this conjecture in the literature. The most significant result in \cite{BurGarPe11} proves that every biorthogonality preserving linear surjection between two von Neumann algebras is automatically continuous. The same conclusion holds in other interesting cases, for example for the algebra of compact operators, but the problem remains open for general C$^*$-algebras. In \cite{BurGarPe11StudiaTriples} we find an extension of this result in the setting of weakly compact JB$^*$-triples and atomic JBW$^*$-triples.\smallskip

Let $N$ be a von Neumann algebra equipped with a normal semi-finite faithful trace $\tau$ and, for any $1\leq p\leq\infty$, let $L^p(N,\tau)$ denote the associated noncommutative $L^p(\tau)$-space. T. Oikhberg and the second author of this note proved in \cite{OikPe13} that if $N$ is a semi-finite factor not of type $I_2$, $1\leq p<\infty$, every linear transformation $T$ from $L^p (N,\tau)$ into a Banach space $X$ which is orthogonality-to-$p$-orthogonality preserving must be necessarily continuous, and if $T\neq 0$, then it is also invertible. It is further established that if $N$ is not a factor, but a separably acting von Neumann algebra, every bijective orthogonality-to-$p$-orthogonality preserving linear mapping $T:L^p(N,\tau)\to X$ is automatically continuous.\smallskip

More results on automatic continuity of bi-separating and biorthogonality preserving linear maps are pursued, for example, in
\cite{AlBreExVill09}, \cite{Bur13}, \cite{BurSanchOr}, \cite{LeuTsaiWong12}, \cite{LeuWong10}, \cite{LiuChouLiaoWong2018b}, 
\cite{TsaiWong10}, \cite{Wong2005}, \cite{Wong2007}. More recently, J.-H. Liu, C.-Y. Chou, C.-J. Liao and N.-C. Wong study bijective linear maps preserving either zero-products or range orthogonality  between AW$^*$-algebras \cite{LiuChouLiaoWong2018}. These maps are shown to be automatically continuous and related to algebra ($^*$-)isomorphisms.\smallskip

Continuous and symmetric linear operators from a unital C$^*$-algebra into another C$^*$-algebra were completely determined by M. Wolff in \cite{Wolff94}. M. Burgos, F.J. Fern{\'a}ndez-Polo, J. Mart{\'i}nez and the authors of this note showed in \cite{BurFerGarMarPe2008} how the study of bounded linear (non-necessarily symmetric) operators between general C$^*$-algebras can benefit from both the language and the techniques of JB$^*$-triple theory. The reader does not need to be familiarized with the theory of JB$^*$-triples in this paper, we simply observe that every C$^*$-algebra $A$ will be equipped with the triple product defined by $$\{a,b,c\} = \frac12 ( ab^* c + c b^* a).$$ A linear mapping between C$^*$-algebras is called a triple homomorphism if it preserves triple products of this form. Building upon the natural structure of JB$^*$-triple underlying each C$^*$-algebra, the description of all bounded linear maps preserving orthogonality between C$^*$-algebras can be summarized in the following theorem.

\begin{theorem}\label{thm characterization of OP}\cite[Theorem 17 and Corollary 18]{BurFerGarMarPe2008} Let $T: A\to B$ be a bounded linear operator between two C$^*$-algebras. For $h= T^{**} (1)$ and $r=r(h)$ the range partial isometry of $h$  (i.e., the partial isometry in the polar decomposition of $h$) in $B^{**}$ the following assertions are equivalent: \begin{enumerate}[{\rm $(a)$}] \item $T$ is orthogonality preserving;
\item There exists a triple homomorphism $S:
A \to B^{**}$ satisfying $h^* S(z) = S(z^*)^* h,$ $h S(z^*)^* = S(z) h^*$ (and consequently $r^* S(z)$ $= S(z^*)^* r$, $ S(z^*)^*$ $= S(z) r^*$) and \begin{equation}\label{eq fund equation conts OP thm 1} T(z) = \frac12 \left( h r^* S(z) + S(z) r^* h\right) = h r^* S(z)=S(z) r^* h,
\end{equation} for all $z\in A$;
\item $T$ preserves zero-triple-products, i.e. $$\{a,b,c\}=0 \Rightarrow \{T(a),T(b),T(c)\}=0.$$\end{enumerate}

\noindent It is further known that under any of the equivalent statements $T(A)$ and $S(A)$ are contained in $B_2^{**} (r) = r r^* B r^* r$. 
\end{theorem}

An extension of this result to the wider setting of orthogonality preserving bounded linear operators from a JB$^*$-algebra into a JB$^*$-triple was obtained in \cite{BurFerGarPe09}.\smallskip

Mathematicians have pursued results guaranteeing the automatic continuity of those linear bijections preserving orthogonality or zero-products in both directions between C$^*$-algebras. However, despite of the abundant literature on orthogonality preserving bounded linear operators, there is no an optimal description of those surjective or bijective bounded linear operators preserving  orthogonality. This paper is aimed to fill this gap, and to apply a concrete description of those maps to determine the uniformly continuous one-parameter semigroups of orthogonality preserving bounded linear operators on a general C$^*$-algebra.\smallskip

Section \ref{sec: surjective OP Cstar algebras} is devoted to throw new light on those bounded and orthogonality preserving linear operators between C$^*$-algebras which are also assumed to be surjective or bijective. By assuming the extra hypothesis of being surjective the conclusion in Theorem \ref{thm characterization of OP} can be sharpened. More concretely, suppose $T: A\to B$ is a bounded linear operator between C$^*$-algebras. Let $h= T^{**} (1)$ and let $r$ denote the range partial isometry of $h$ in $B^{**}$. In Theorem \ref{t Characterization bd OP plus surjective} we prove that the following statements are equivalent:\begin{enumerate}[$(a)$] \item $T$ is surjective and orthogonality preserving;
\item The elements $h$ and $r$ belong to the multipliers algebra of $B$ with $h$ invertible and $r$ unitary and there exists a surjective triple homomorphism $S$ from $A$ into $B$ satisfying $h^* S(x) = S(x^*)^* h$, $h S(x^*)^* = S(x) h^*$, $h r^* S(x) = S(x) r^* h,$ and $T(x) = h r^* S(x) = S(x) r^* h,$ for all $x\in A$;
\item The elements $h$ and $r$ belong to the multipliers algebra of $B$ with $h$ invertible and $r$ unitary and there exist another C$^*$-algebra structure on $B,$ $(B,\bullet_{r},*_{r}),$ and a surjective Jordan $^*$-homomorphism $S: A\to (B,\bullet_{r},*_{r})$ such that $h$ lies in the center of $(B^{**},\bullet_{r},*_{r})$ 
     and $$T(x) = h r^* S(x) = S(x) r^* h, \hbox{ for all } x\in A;$$
\item $\ker(T)$ is a norm closed ideal of $A$ and the quotient map $\widehat{T}: A/\ker(T)\to B$ is (continuous) surjective and orthogonality preserving;
\end{enumerate}

Corollary \ref{c Characterization bd OP plus bijective} is worth to be commented by itself, in this result we show that if $T: A\to B$ is a bijective bounded linear operator between C$^*$-algebras, 
then the following statements are equivalent:\begin{enumerate}[$(i)$] \item $T$ is orthogonality preserving (or it simply preserves orthogonality on $A_{sa}$, or on $A^{+}$);
\item $T$ is biorthogonality preserving (or it simply preserves orthogonality in both directions on $A_{sa}$, or on $A^{+}$);
\item $T$ preserves zero-triple-products, i.e. $$\{a,b,c\}=0 \Rightarrow \{T(a),T(b),T(c)\}=0;$$
\item $T$ preserves zero-triple-products in both directions, i.e. $\{a,b,c\}=0 \Leftrightarrow \{T(a),T(b),T(c)\}=0.$
\end{enumerate}

M. Wolff found in \cite[Theorem 2.6]{Wolff94} a detailed description of all uniformly continuous one-parameter semigroups of symmetric orthogonality preserving bounded operators on an arbitrary unital C$^*$-algebra (detailed definitions can be found in section \ref{sec: one-parameter semigroups Cstar algebra}). We shall establish different generalizations of Wolff's theorem in several directions in the third section of this paper.\smallskip

Our conclusion on uniformly continuous one-parameter semigroups of (non-necessarily symmetric) orthogonality preserving bounded operators (see Theorem \ref{t Wolff one-parameter for OP}) affirms that given a family $\{T_t: t\in \mathbb{R}_0^{+}\}$ of orthogonality preserving bounded linear bijections on a general C$^*$-algebra $A$ with $T_0=Id$, if for each $t\geq 0,$ we set $h_t = T_t^{**} (1)$ and we write $r_t$ for the range partial isometry of $h_t$ in $A^{**}$ and, by virtue of Corollary \ref{c Characterization bd OP plus bijective}, $S_t$ stands for the triple isomorphism on $A$ associated with $T_t$ satisfying $h_t^* S_t(x)$  $= S_t(x^*)^* h_t$, $h_t S_t(x^*)^* =$ $ S_t(x) h_t^*$, $h_t r_t^* S_t(x) =$  $S_t(x) r_t^* h_t$, and $$T_t(x) = h_t r_t^* S_t(x) = S_t(x) r_t^* h_t, \hbox{ for all } x\in A,$$ the following statements are equivalent:\begin{enumerate}[$(a)$]\item $\{T_t: t\in \mathbb{R}_0^{+}\}$ is a uniformly continuous  one-parameter semigroup of orthogonality preserving operators on $A$;
\item $\{S_t: t\in \mathbb{R}_0^{+}\}$ is a uniformly continuous one-parameter semigroup of surjective linear isometries (i.e. triple isomorphisms) on $A$ (and hence there exists a triple derivation $\delta$ on $A$ such that $S_t = e^{t \delta}$ for all $t\in \mathbb{R}$), the mapping $t\mapsto h_t $ is continuous at zero, and the identity $ h_{t+s} =  h_t r_t^* S_t^{**} (h_s),$ holds for all $s,t\in \mathbb{R}.$
\end{enumerate}

A generalization of Wolff's theorem \cite[Theorem 2.6]{Wolff94} for one-parameter semigroups of orthogonality preserving symmetric operators on non-unital C$^*$-algebras is established in Corollary \ref{c Wolff one-parameter for OP symmetric}, in this particular case the sets $\{r_t:t\in \mathbb{R}\}$ and  $\{h_t:t\in \mathbb{R}\}$ are shown to be one-parameter groups in the center of the multipliers algebra, $M(A),$ of $A$. However, we shall show in Remark \ref{R the ht are not a group in the general case} that in the general setting the sets $\{r_t:t\in \mathbb{R}^+_0\}$ and $\{h_t:t\in \mathbb{R}^+_0\}$ need not be one-parameter semigroups.
\smallskip

As we have previously commented, a description of all orthogonality preserving bounded\hyphenation{bound-ed} linear operator on a JB$^*$-algebra is obtained in \cite{BurFerGarPe09}. The question whether the results contained in this paper can be generalized to the setting of JB$^*$-algebras arises naturally. This question will be treated in the forthcoming paper \cite{GarPeUnitJBstaralg20} through new and independent techniques and arguments from the theory of JB$^*$-algebras and JB$^*$-triples.

\subsection{Notation and definitions}\label{subsec: notation and defs}

Throughout this note, given a C$^*$-algebra $A$ the symbols $A_{sa}$ and $A^{+}$ will denote the self-adjoint part and the positive cone of $A,$ respectively. Given $a\in A$, we denote by $L_a$, $R_a$, and $M_a$ the mappings given by $L_a(x) =  ax$, $R_a(x) = x a$, and $M_a = \frac12 (L_a + R_a)$, respectively.\smallskip

As we commented in the introduction, the study of (non-necessarily symmetric) bounded linear operators preserving orthogonality between general C$^*$-algebras is more accesible if each C$^*$-algebra $A$ is endowed with its natural triple product
$\J xyz := \frac{1}2 (xy^*z + zy^*x),$ ($x,y,z\in A$).\smallskip

Each C$^*$-algebra $A$ can be equipped with its natural Jordan product defined by $a\circ b=\frac{1}{2}(ab+ba)$. A \emph{Jordan homomorphism} between C$^*$-algebras $A$ and $B$ is a linear mapping $T:A\to B$ satisfying $T(a\circ b)=T(a)\circ T(b), $ for all $a,b\in A.$ A \emph{Jordan $^*$-homomorphism} is a {Jordan homomorphism} $T$ satisfying $T(a^*) = T(a)^*$ for all $a\in A$.  A linear mapping $T:A\to B$ is said to be a \emph{triple homomorphism} if  $T\{a,b,c\} = \{T(a),T(b), T(c)\}$, for all $a,b,c\in A$.\smallskip

A norm-closed linear subspace $ I$ of $A$ is said to be a \emph{subtriple} if $\J III\subseteq I$. Each partial isometry $e$ (i.e. $ee^* e= e$) in a C$^*$-algebra $A$ induces a \emph{Peirce decomposition} of $A$ as the direct sum of the subtriples
$$A_2 (e) = ee^* A e^* e,  \ A_1 (e) = ee^* A (1-e^*e) \oplus (1-ee^*) A e^*e,$$ $$\hbox{ and }  A_0 (e) = (1-ee^*) A (1-e^*e).$$ The Peirce projections of $A$ onto the Peirce subspaces are given by $P_2 (e) (x) = e e^* x e^*e$,  $P_1 (e) (x) = (1- e e^*) x e^*e + e e^* x (1-e^*e)$, and $P_0 (e) (x) = (1-e e^*) x (1-e^*e).$ The Peirce 2-subspace $A_2 (e)$ is a C$^*$-algebra with  product $x\bullet_e y :=  xe^*y$ and involution $x^{*_e} := \J exe = e x^* e$ (\cite{Zettl}). Thus the Jordan product in $A_2(e)$ is given by $x\circ_e y=\J xey = \frac12 (x\bullet_e y + y\bullet_e x)$.\smallskip

Let $a$ be an element in a C$^*$-algebra $A$. Suppose $a = u |a|$ is the polar decomposition of $a,$ where $|a|= (a^* a)^{\frac12}$ and $u$ is the unique partial isometry in $A^{**}$ such that $u^* u$ is the range projection of $|a|$ in $A^{**},$ which is also denoted by $r(|a|)$ (cf. \cite[\S 1.12 and Definition 1.10.3]{Sa}, \cite[Proposition 2.2.9]{Ped}). The partial isometry $u$ is known as the \emph{range partial isometry} of $a$ in $A^{**}$, and it is also denoted by $r(a)$.\smallskip

If we consider the ``odd triple'' powers of $a$ defined by $a^{[1]}= a$, $a^{[3]} = \J aaa$, and $a^{[2n+1]} := \J aa{a^{[2n-1]}},$ $(n\in \NN)$, for each odd polynomial (with zero constant term) $p(t)=\sum_{i=1}^n\lambda_i t^{2i-1}$, we set $p_t(a):=\sum_{i=1}^n\lambda_i a^{[2i-1]}$. Seeking for a connection with the continuous functional calculus, we see that $a^{[3]} = u |a| u^* u |a|^2 = u |a|^3,$ and by induction, $a^{[2n -1]} =  u |a|^{2n -1},$ for all $n\in \mathbb{N}$. Therefore $p_t(a) = u p(|a|)$ for every odd polynomial $p(\lambda)$ with zero constant term, where the right-hand-side to the C$^*$- continuous functional calculus of $p(\lambda)$ at $|a|$. A similar procedure was already employed by C.A. Akemann and G.K. Pedersen in \cite{AkPed77}. Namely, let $f$ be a continuous complex-valued function on the spectrum of $|a|$ which vanishes at zero. The element $f_t (a) := u f(|a|)$ lies in $A$ (cf. \cite[Lemma 2.1]{AkPed77}). In coherence with the usual notation in a wider setting, the element $f_t(a)$ will be called the \emph{continuous triple functional calculus} of $f$ at $a$. The triple functional calculus enjoys some additional properties, for example, every $a\in A$ admits a (unique) \emph{cubic root}, that is, a unique element $a^{[\frac{1}{3}]} = f_t (a) = u |a|^{\frac13},$ where $u$ is the range partial isometry of $a$ in $A^{**}$, such that $\{a^{[\frac{1}{3}]},a^{[\frac{1}{3}]},a^{[\frac{1}{3}]}\}=a.$ We can inductively define the sequence $( a^{[\frac{1}{3^n}]})_n $ by  $( a^{[\frac{1}{3^{n+1}}]})= \big( a^{[\frac{1}{3^n}]}\big)^{[\frac{1}{3}]}$. From the last paragraph we can assure that $a^{[\frac{1}{3^n}]} = u |a|^{\frac{1}{3^n}}$, for all natural $n$. Taking weak$^*$ limits in $A^{**}$ we get that the sequence $(a^{[\frac{1}{3^n}]})$ converges to the partial isometry  $ u r(|a|) = u = r(a)$.\smallskip

One of the characterizations of von Neumann algebras proves that a C$^*$-algebra is a von Neumann algebra precisely when it is a dual Banach space (cf.  \cite[Theorem 3.9.8]{Ped} or \cite[\S 1.1 and \S 1.20]{Sa} or \cite[Definition II.3.2, Theorem III.3.5]{Tak}). One of the most famous theorems due to S. Sakai affirms that the product of every von Neumann algebra is separately weak$^*$ continuous \cite[Theorem 1.7.8]{Sa}. Each von Neumann algebra admits a topology providing better properties for the product. Suppose $\varphi$ is a positive functional in the predual of a von Neumann algebra $M$. The mapping $x\mapsto \|x\|_{\varphi} =\varphi(xx^* + x^* x)^{\frac12}$ is a preHilbertian seminorm on $M$. The \emph{strong$^*$ topology} of $M$ is the topology generated by all the seminorms $\|\cdot\|_{\varphi}$ with $\varphi$ running in the set of positive normal functionals on $M$ (see \cite[Definition 1.8.7]{Sa}). The strong$^*$ topology is stronger than the weak$^*$ topology of $M,$ and the product of $M$ is jointly strong$^*$ continuous on bounded sets \cite[Theorem 1.8.9]{Sa}. A linear mapping $T$ between von Neumann algebras is weak$^*$-to-weak$^*$ continuous if and only if it is strong$^*$-to-strong$^*$ continuous (see \cite[Corollary 1.8.10]{Sa}). It should be noted that, for each $a$ in a C$^*$-algebra $A$, the sequence $(a^{[\frac{1}{3^n}]})$ also converges to $r(a)$ in the strong$^*$ topology of $A^{**}$.\smallskip

A \emph{two-sided ideal}, or simply an \emph{ideal}, of a C$^*$-algebra $A$ is a subspace $I$ satisfying $I A, A I \subseteq I$.  We say that a (norm closed) subspace $I$ of $A$ is a closed Jordan (respectively, triple) ideal of $A$ if $I\circ A \subseteq I$ (respectively, if $\{A,A,I\}\subseteq I$) holds. By \cite[Proposition 5.8]{Harris81} closed two-sided, Jordan and triple ideals in a C$^*$-algebra all coincide, it is further known that they are all self-adjoint.

\section{Surjective Orthogonality preserving operators on C$^*$-algebras}\label{sec: surjective OP Cstar algebras}

Let $T:A \to B$ be a (linear) mapping between C$^*$-algebras. We shall say that $T$ is orthogonality preserving on a set $\mathcal{S}\subseteq A$ if $a\perp b$ in $\mathcal{S}$ implies $T(a) \perp T(b)$.
 Theorem 20 in \cite{BurFerGarMarPe2008} shows that a bounded linear operator $T:A\to B$ between two C$^*$-algebras is a triple homomorphism if, and only if, $T$ preserves orthogonality on $A_{sa}$ and $T^{**}(1)$ is a partial isometry. It is therefore natural to ask whether in Theorem \ref{thm characterization of OP} we may only assume that $T$ preservers orthogonality on $A_{sa}.$ We show next that this is always the case.

\begin{proposition}\label{r OP on Asa} Let $T: A\to B$ be a bounded linear operator between C$^*$-algebras. The following statements are equivalent:
\begin{enumerate}[$(OP1)$]\item $T$ is orthogonality preserving;
\item $T$ is orthogonality preserving on $A_{sa}$;
\item $T$ is orthogonality preserving on $A^{+}$.
\end{enumerate}
\end{proposition}



\begin{proof}
The implications $(OP1)\Rightarrow (OP2)\Rightarrow (OP3)$ are clear. Let us show that $(OP3)\Rightarrow (OP1)$. We fix an arbitrary $\phi\in B^{*}$ and define two bilinear forms on $A$ defined by $V_{1,\phi} (a,b) := \phi (T(a) T(b^*)^*)$ and $V_{2,\phi} (a,b) := \phi (T(b^*)^* T(a))$. Let us take $a,b\in A_{sa}$ with $a\perp b$. It is easy to see from Gelfand theory that we can find four mutually orthogonal positive elements $a_1,a_2,b_1,b_2\in A$ such that $a = a_1-a_2$ and $b=b_1-b_2$. By assumptions, $T(a_i) \perp T(b_j)$ for all $i,j=1,2$. Therefore, $T(a_i) T(b_j^*)^* = T(b_j^*)^*  T(a_i) =0$ for all $i,j\in \{1,2\}$, and thus $V_{1,\phi} (a,b) = V_{2,\phi} (a,b)=0$, witnessing that $V_{1,\phi}$ and $V_{2,\phi}$ are orthogonal forms in Goldstein's sense \cite{Gold}. By Theorem 1.10 in \cite{Gold} there exist $\tau_1,\tau_2,\rho_1,\rho_2 \in A^*$ such that
$$
V_{1,\phi}(x,y)=\tau_1(xy)+\tau_2(y x) \mbox{ and } V_{2,\phi}(x,y)=\rho_1(xy)+\rho_2(y x),\hbox{ for all $x,y\in A.$}
$$
 Obviously $\phi (T(x) T(y)^*) = V_{1,\phi}(x,y^*) = 0 = V_{2,\phi}(x,y^*) =\phi (T(y)^* T(x))$ for all $x,y\in A$ with $x\perp y.$ The Hahn-Banach theorem implies that $T(x)\perp T(y)$ for all $x\perp y$ in $A$. \smallskip
\end{proof}

The properties of the kernel of an orthogonality preserving bounded linear operator between C$^*$-algebras are studied in the next lemma.

\begin{lemma}\label{l kernel is an ideal} Let $T : A\to B$ be an orthogonality preserving bounded linear operator between two C$^*$-algebras. Let $h=T^{**} (1)\in B^{**}$, $r$ the range partial isometry of $h$ in $B^{**}$, and let $S: A \to B^{**}$ be the triple homomorphism given by Theorem \ref{thm characterization of OP}. Then $\ker(T) =\ker (S)$ is a norm closed ideal of $A$.
\end{lemma}

\begin{proof} It is almost explicit in the proof of \cite[Theorem 17 and Corollary 18]{BurFerGarMarPe2008} that $S(A)\subseteq B^{**}_2(r) = rr^* B^{**} r^* r$ and $S^{**} (1) = r$ (see also \cite[Theorem 4.1]{BurFerGarPe09}, where a more explicit argument is given in the wider setting of JB$^*$-algebras).\smallskip

Obviously $\ker (S) \subseteq \ker (T)$ (cf. Theorem \ref{thm characterization of OP}). Take now $x\in \ker (T)$. That is, $0 = T(x) = h r^* S(x)$. It then follows that $$h^{[3]} r^* S(x) = hh^*h r^* S(x)= 0,$$ and by induction $r |h|^{2n-1} r^* S(x)= h^{[2n-1]} r^* S(x) =  0,$ for all natural $n$. Consequently $r z r^* S(x) =0$ for all $z$ in the C$^*$-subalgebra of $B^{**}$ generated by $|h|$. Since $r^* r = r(|h|)$ lies in the weak$^*$ closure of the C$^*$-subalgebra of $B^{**}$ generated by $|h|$, and the product of $B^{**}$ is separately weak$^*$ continuous (cf. \cite[Theorem 1.7.8]{Sa}), we can easily deduce that $0= r r^* S(x),$ which gives $S(x)=0$ because $S(A)\subseteq B^{**}_2(r) = rr^* B^{**} r^* r$. We have therefore proved that $\ker (T) = \ker (S)$ is a norm closed triple ideal of $A$. Thus, $ker(T)$ is a closed (self-adjoint) ideal of $A$ (cf. \cite[Corollary 5.8]{Harris81}).
\end{proof}

We are now in a position to reveal the nature of all surjective bounded linear operators preserving orthogonality.

\begin{proposition}\label{p surjective OP operators} Let $T : A\to B$ be a continuous and surjective linear operator preserving orthogonality between two C$^*$-algebras. Let $h=T^{**} (1)\in B^{**}$ and let $r$ denote the range partial isometry of $h$ in $B^{**}$. Then the following statements hold:
\begin{enumerate}[$(a)$] \item $r$ is a unitary element in $B^{**}$;
\item $\ker(T)$ is a norm closed (self-adjoint) ideal of $A$;
\item If $T$ is bijective the element $h$ is invertible in $B^{**}$;
\item The quotient mapping $\widehat{T}: A/\ker(T)\to B,$ $\widehat{T}(x+\ker(T)) = T(x)$ is a bijective orthogonality preserving bounded linear mapping;
\item $h$ is invertible in $B^{**}$,
\item There exist a triple homomorphism $S: A\to B^{**}$ and a triple monomorphism $\widehat{S}: A/\ker(S)\to B^{**}$ satisfying \begin{enumerate}[$(1)$] \item $\ker(T)= \ker(S)$;
    \item $\widehat{S}^{**} (1+\overline{\ker(S)}^{w^*}) =S^{**} (1) = r;$
    \item $\widehat{S} (x+{\ker(S)}) =S(x)$;
    \item $h^* S(x) = S(x^*)^* h$, $h S(x^*)^* = S(x) h^*$ (and hence $r^* S(x) = S(x^*)^* r$, $r S(x^*)^* = S(x) r^*$), $h r^* S(x) = S(x) r^* h$;
    \item $\widehat{T}(x+\ker(T)) =T(x) = h r^* S(x) = S(x) r^* h,$ for all $x\in A$.
\end{enumerate}
\end{enumerate}
\end{proposition}

\begin{proof}$(a)$ We know from \cite[$(19)$ in the proof of Theorem 17]{BurFerGarMarPe2008} that $T(A)\subseteq B^{**}_2 (r) = rr^* B^{**} r^* r$. Since, by hypothesis $T$ is surjective, we deduce that $B \subseteq B^{**}_2 (r) = rr^* B^{**} r^*r$. Having in mind that, by Goldstine's theorem, $B$ is weak$^*$ dense in $B^{**}$ and $B^{**}_2 (r)$ is weak$^*$-closed, it follows that $B^{**} = B^{**}_2 (r)$ which proves that $r$ is a unitary in $B^{**}$.\smallskip

$(b)$ The desired conclusion follows from Lemma \ref{l kernel is an ideal} even under weaker hypotheses.\smallskip

$(c)$ Suppose $T$ is bijective and hence an isomorphism. Let $S: A\to B^{**}$ be the triple homomorphism given by Theorem \ref{thm characterization of OP}. The bitranspose $T^{**} : A^{**}\to B^{**}$ must be an isomorphism too, and consequently there exists $z\in A^{**}$ satisfying $1 = T^{**} (z).$  The element $z$ can be achieved as the weak$^*$ limit of a bounded net $(a_{\lambda}) $ in $A$. Since, for example, by \cite[Proposition 3.4]{Harris81}, $S$ is contractive the net $(S(a_{\lambda}))$ is bounded. By Banach-Alouglu's theorem  $(S(a_{\lambda}))$ admits a subnet  $(S(a_{\mu}))$ which is convergent in the weak$^*$ topology of $B^{**}$ to some $c\in B^{**}.$ Let us note that the subnet $(T(a_{\mu}))$ converges to $T^{**} (z) =1$ in the weak$^*$ topology of $B^{**}$. It follows from the weak$^*$ continuity of $T^{**}$, the separate weak$^*$ continuity of the product of $B^{**}$ and the identity in \eqref{eq fund equation conts OP thm 1} that $$1 = T^{**} (z) = h r^* c = c r^* h,$$ witnessing that $h$ is invertible in $B^{**}$.\smallskip

$(d)$ The mapping $\widehat{T}: A/\ker(T)\to B,$ $\widehat{T}(x+\ker(T)) = T(x)$ is a well-defined bounded linear bijection. Suppose that $x+\ker(T), y +\ker(T)$ are two self-adjoint elements in the quotient $A/\ker(T)$ with $x+ker(T)\perp y+ker(T)$. We can assume that $x,y\in A_{sa}$. Corollary 2.4 in \cite{AkPed77} implies the existence of $a,b\in \ker(T)$ such that $(x-a)\perp (y-b)$. By hypothesis $\widehat{T}(x+\ker(T)) = T(x-a)\perp T(y-b) = \widehat{T}(y+\ker(T))$, which proves that $\widehat{T}$ is orthogonality preserving on $(A/\ker(T))_{sa}$. We deduce from Proposition \ref{r OP on Asa} that $\widehat{T}$ is orthogonality preserving. \smallskip

$(e)$ We have seen in the proof of $(b)$ that $\ker(T) = \ker(S)$. We know from $(d)$ that $\widehat{T}: A/\ker(T)\to B,$ $\widehat{T}(x+\ker(T)) = T(x)$ is a bounded linear bijection preserving orthogonality. Now, applying $(c)$ it follows that $h= T^{**} (1) = \widehat{T}^{**} (1+\overline{\ker(T)}^{w^*})$ is invertible in $B^{**}$.\smallskip

Statement $(f)$ is clear from the previous ones.
\end{proof}

Let us observe that the conclusion in Proposition \ref{p surjective OP operators}$(f)$ is very close to be a characterization of surjective orthogonality preserving linear operators between C$^*$-algebras. To get a more precise statement we need to work a bit more. The next lemma is straight consequence of \cite[Theorem 17 and Corollary 18]{BurFerGarMarPe2008} (see Theorem \ref{thm characterization of OP}).\smallskip

\begin{lemma}\label{l equation for multiplier} Let $A$ and $B$ be C$^*$-algebras and let $T:A\to B$ be a bounded linear operator preserving orthogonality. Then the identity \begin{equation}\label{eq identity multiplier} \{T(a), T(b), T(c)\} = h^{[3]} r^* S(\{a,b,c\}),
\end{equation} holds for all $a,b,c$ in $A$.
\end{lemma}

\begin{proof} Theorem \ref{thm characterization of OP} \cite[Theorem 17 and Corollary 18]{BurFerGarMarPe2008} guarantees the existence of a triple homomorphism $S:A\to B^{**}$ satisfying $T(A),S(A)\subseteq B_2^{**}(r)= rr^* B^{**} r^* r$, $h^* S(z) = S(z^*)^* h,$ $h S(z^*)^* = S(z) h^*$ (and hence $r^* S(z) = S(z^*)^* r$, $r S(z^*)^* = S(z) r^*$) and $T(z) = h r^* S(z)=S(z) r^* h,$ for all $z\in A$. Therefore $$\begin{aligned}\{ T(a), T(b), T(a)\} &= h r^* S(a) S(b)^* r h^* h r^* S(a) = h r^* S(a) h^* r S(b)^* h r^* S(a) \\
&= h r^* h S(a^*)^* r h^* S(b^*) r^* S(a) = h r^* h r^* S(a) h^* r S(b)^* S(a) \\
&= h r^* h r^* h r^* S(a) S(b)^* S(a) = h h^* r r^* h r^* \{S(a), S(b), S(a)\} \\
&= \{h,h,h\} r^* S(\{a,b,a\}) = h^{[3]} r^* S(\{a,b,a\}),
\end{aligned}$$ where we employed that $r^* h = h^* r$ and $r r^* h =h $ because $r$ is the range partial isometry of $h$ in $B^{**}$. The desired identity follows from the symmetry of the triple product in the outer variables.
\end{proof}

We recall that the \emph{multipliers algebra}, $M(A),$ of a C$^*$-algebra $A$, is the C$^*$-subalgebra of $A^{**}$ of all
elements $x\in A^{**}$ such that $x a, a x \in A$ for all $a\in A$ (cf. \cite[Definition III.6.22]{Tak}). If $A$ is unital this extension gives no additional information as $M(A) =A$. 
\smallskip

In \cite{BuChu92}  L. Bunce and Ch.-H. Chu extend the concept of multipliers to the more general category of JB$^*$-triples. Among the consequences of their results it is shown that the multipliers algebra of a C$^*$-algebra $A$ can be characterised in terms of its triple product. More concretely, if $A$ is endowed with its natural triple product we have  $$M(A) =\{ x\in A^{**} : \{x,A,A\}\subseteq A\}  \hbox{\ \  (cf. \cite[page 253]{BuChu92}).} $$ The result is explicitly proved in \cite{BuChu92}, however the equality can be checked without any reference to JB$^*$-triple theory.

\begin{proposition}\label{p surjective OP operators connected to multipliers} Let $T : A\to B$ be a surjective bounded linear operator preserving orthogonality between two C$^*$-algebras. Let $h=T^{**} (1)\in B^{**}$ and let $r$ denote the range partial isometry of $h$ in $B^{**}$. Let $S: A\to B^{**}$ be the triple homomorphism given by Theorem \ref{thm characterization of OP}. Then the following statements hold:
\begin{enumerate}[$(a)$] \item The elements $h$ and $r$ belong to $M(B)$;
\item The triple homomorphism $S$ is $B$-valued and surjective;
\item $T^{**}(M(A))\subseteq M(B)$ and $S^{**}(M(A))\subseteq M(B)$.
\end{enumerate}
\end{proposition}

\begin{proof} $(a)$ By Lemma \ref{l equation for multiplier} the identity \begin{equation}\label{eq identity from lemma 6} \{T(a), T(b), T(c)\} = h^{[3]} r^* S(\{a,b,c\}), \end{equation} holds for all $a,b,c\in A$. Since the left-hand-side in \eqref{eq identity from lemma 6} lies in $B$, we deduce that \begin{equation}\label{eq products of T in B}\hbox{ the products of the form $h^{[3]} r^* S(\{a,b,c\})$ lie in $B$ for all $a,b,c\in A$.}
 \end{equation}

Take a bounded approximate unit $(u_{\lambda})_{\lambda}$ in $A$ (see \cite[Theorem 1.4.2]{Ped}). It is known that $(u_{\lambda})_{\lambda}\to 1$ in the weak$^*$ topology of $A^{**}$. Fix arbitrary $b,c\in A$. By replacing $a$ with $u_{\lambda}$ in \eqref{eq identity from lemma 6}, taking weak$^*$ limits in $\lambda$ on the left hand side and applying the separate weak$^*$-continuity of the product in $A^{**}$, we get $$\{h, T(b), T(c)\} =\{T^{**} (1), T(b), T(c)\} =\hbox{w$^*$-}\lim_{\lambda} \{T(u_{\lambda}), T(b), T(c)\}.$$ However, after replacing $a$ with $u_{\lambda}$ in \eqref{eq identity from lemma 6} on the right hand side we apply that $(\{u_{\lambda},b,c\})\to \{1,b,c\} = \frac12 (b^* c + c b^*)\in A$ in norm. Therefore, $$ h^{[3]} r^* S\left(\frac{b^*c + cb^*}{2}\right) = \|\cdot\|\hbox{-}\lim_{\lambda} h^{[3]} r^* S(\{u_{\lambda},b,c\}),$$ which proves that \begin{equation}\label{eq b c arbitrary 1707} \{h, T(b), T(c)\} =  h^{[3]} r^* S\left(\frac{b^*c + cb^*}{2}\right),
 \end{equation} for all $ b,c\in A$.\smallskip

By applying the existence of cubic root for the element $\frac12 (b^* c + c b^*)\in A,$ we find $y \in A$ such that $\{y,y,y\} = \frac12 (b^* c + c b^*)$. Therefore, by \eqref{eq b c arbitrary 1707} and \eqref{eq products of T in B},  $$\{h, T(b), T(c)\} = h^{[3]} r^* S(\{y,y,y\}) \in B,$$ for all $b,c\in A$. Now, applying that $T$ is surjective we conclude that $\{h, B, B\} \subseteq B$, that is, $h$ lies in the (triple) multipliers of $B$ in $B^{**}$.\smallskip

We shall next show that \begin{equation}\label{eq r*S is B valued} r^* S(A)\subseteq B.
\end{equation} Namely, by Proposition \ref{p surjective OP operators}$(a)$ and $(e)$, $r$ is a unitary and $h$ is invertible in $B^{**}$. We know from the above that $h^{-1}\in M(B)$. Theorem \ref{thm characterization of OP} assures that $T(x) = h r^* S(x)$ for all $x\in A$. Since for each $a\in A,$ $r^*S(a)  = h^{-1} T(a)\in B$ for all $a\in A$, the statement in \eqref{eq r*S is B valued} holds.\smallskip

We claim that $h r^*\in M(B)$. By the surjectivity of $T$, for each $b\in B$ there exists $a\in A$ satisfying $T(a) =b$. The element $r^* S(a)\in B$ (see \eqref{eq r*S is B valued}), and thus, again by the surjectivity of $T$, there exists $c\in A$ such that $T(c) = r^* S(a)$. Therefore, having in mind that $r$ is the range partial isometry of $h$ we have  $$ h r^* b = h r^* T(a) = h r^* h r^* S(a) =  h h^* r r^* S(a) =  h h^* r T(c)  $$
$$= |h^*| |h^*| r T(c) = |h^*| (r^* |h^*|)^* T(c) = |h^*| (h^*)^* T(c) = |h^*| h T(c) \in B,$$ because $T(c)\in B$ and $h\in M(B)$. This concludes the proof of the claim.\smallskip

Since $r^* = h^{-1} h r^*$ with $h^{-1}, h r^*\in M(B)$, the element $r\in M(B)$.\smallskip

$(b)$ Having in mind that $r$ is unitary, $h$ invertible in $M(B)$ and the conclusion of Theorem \ref{thm characterization of OP} we easily get that $S(a) = r h^{-1} T(a)\in B$, for all $a\in A$. The latter identity combined with the surjectivity of $T$ assures the surjectivity of $S$.\smallskip

$(c)$ Let us observe that as a consequence of $(b)$ the triple homomorphism $S^{**}$ is $B^{**}$-valued. By the separate weak$^*$ continuity of the product of every von Neumann algebra and Goldstine's theorem, the identity in \eqref{eq identity from lemma 6} holds for all $a,b,c\in A^{**}$ by just replacing $T$ and $S$ with $T^{**}$ and $S^{**}$, respectively. Similarly, we have $T^{**}(a)=hr^*S^{**}(a)$ for all $a\in A^{**}$. Let us fix $a\in M(A)$. It follows from just commented properties that $$\{T^{**}(a), T(b), T(c)\} = h^{[3]} r^* S(\{a,b,c\})\in B,$$ for all $b,c\in B,$ because $a, r, h\in M(A)$. Having in ming that $T$ is surjective we deduce that $\{T^{**}(a), B, B\} \subseteq B$, that is, $T^{**}(a)$ lies in the (triple) multipliers of $B$ in $B^{**}$, or equivalently, $T^{**} (a)\in M(B)$. The second inclusion in $(c)$ is now clear from the identity $S^{**} (a) = r h^{-1} T^{**} (a)$ for all $a\in A^{**}$.
\end{proof}

Henceforth the center of a C$^*$-algebra $A$ will be denoted by $Z(A)$, and $A^{-1}$ will stand for the open subgroup of all invertible elements in $A$.

\begin{remark}\label{r invertibility and center in the homotope} Let $r$ be a unitary element in the multipliers algebra, $M(A)$, of a C$^*$-algebra $A$. As in previous pages, we consider the C$^*$-algebra $(A,\bullet_r, *_{r})$, where $a\bullet_r b = a r^* b$ and $a^{*_r} = r a^* r$. It is easy to check that the following equalities hold: $$r Z(A) = Z(A) r = Z(A,\bullet_r, *_{r}),\hbox{ and } A^{-1} = (A,\bullet_r, *_{r})^{-1}.$$ Namely, if $z\in Z(A)$ we have $z r= r z$ and  $(zr) \bullet_r a = z r r^* a = z a = a z = a r^* r z = a \bullet_r (zr)$. If $b\in A^{-1}$ with inverse $b^{-1}$, $b$ is also invertible in $(A,\bullet_r, *_{r})$ with inverse $r b^{-1} r$.\smallskip

Let $\Phi : A\to B$ be a Jordan $^*$-isomorphism, where $B$ is another C$^*$-algebra. Since $\Phi^{**}: A^{**}\to B^{**}$ also is a Jordan $^*$-isomorphism, a celebrated result due to R.V. Kadison asserts that $\Phi^{**}$ decomposes as a direct sum of a $^*$-isomorphism and a $^*$-anti-homomorphism (cf. \cite[Theorem 10]{Kad51}). Consequently, $\Phi (Z(A)) = Z(B)$. Moreover, if $a\in Z(A)$ we have
\begin{equation}\label{eq element in the center is like a homom} \Phi(a b) = \Phi (a\circ b) = \Phi(a) \circ \Phi(b) = \Phi(a) \Phi(b), \hbox{  for all } b\in A.
\end{equation}
\end{remark}

We can now state a complete characterization of those bounded linear operators between C$^*$-algebras which are orthogonality preserving and surjective.

\begin{theorem}\label{t Characterization bd OP plus surjective} Let $T: A\to B$ be a bounded linear operator between C$^*$-algebras. Let $h= T^{**} (1)$ and let $r$ denote the range partial isometry of $h$ in $B^{**}$. The following statements are equivalent:\begin{enumerate}[$(a)$] \item $T$ is surjective and orthogonality preserving;
\item The elements $h$ and $r$ belong to $M(B)$ with $h$ invertible and $r$ unitary and there exists a surjective triple homomorphism $S: A\to B$ satisfying $h^* S(x) = S(x^*)^* h$, $h S(x^*)^* = S(x) h^*$ (and hence $r^* S(x) = S(x^*)^* r$, $r S(x^*)^* = S(x) r^*$), $h r^* S(x) = S(x) r^* h,$ and $T(x) = h r^* S(x) = S(x) r^* h,$ for all $x\in A$;
\item The elements $h$ and $r$ belong to $M(B)$ with $h$ invertible and $r$ unitary and there exists a surjective Jordan $^*$-homomorphism $S: A\to (B,\bullet_{r},*_{r})$ such that $h\in Z(B^{**},\bullet_{r},*_{r}) = r Z(B^{**})= Z(B^{**}) r$,
    and $$T(x) = h r^* S(x) = S(x) r^* h, \hbox{ for all } x\in A;$$
\item $\ker(T)$ is a norm closed ideal of $A$ and the quotient map $\widehat{T}: A/\ker(T)\to B$ is (continuous) surjective and orthogonality preserving.
\end{enumerate}
\end{theorem}

\begin{proof} The equivalences are clear consequences of the result stated in Propositions \ref{p surjective OP operators} and \ref{p surjective OP operators connected to multipliers} and Remark \ref{r invertibility and center in the homotope}. Let us simply observe that $h \bullet_r S(a) = S(a) \bullet_r h$ for all $a\in A$. It follows from the surjectivity of $S$ and Goldstine's theorem that $h\in  Z(B^{**},\bullet_{r},*_{r})= r Z(B^{**}) = Z(B^{**}) r$.
\end{proof}

The next corollary is worth to be considered by itself.

\begin{corollary}\label{c Characterization bd OP plus bijective} Let $T: A\to B$ be a bijective bounded linear operator between C$^*$-algebras. Let $h= T^{**} (1)$ and let $r$ denote the range partial isometry of $h$ in $B^{**}$. The following statements are equivalent:\begin{enumerate}[$(a)$] \item $T$ is orthogonality preserving;
\item The elements $h$ and $r$ belong to $M(B)$ with $h$ invertible and $r$ unitary, and there exists a triple isomorphism $S: A\to B$ satisfying $h^* S(x)$  $= S(x^*)^* h$, $h S(x^*)^* =$ $ S(x) h^*$ (and consequently $r^* S(x)$ $= S(x^*)^* r$, $r S(x^*)^* =$ $ S(x) r^*$), $h r^* S(x) =$  $S(x) r^* h$, and $$T(x) = h r^* S(x) = S(x) r^* h, \hbox{ for all } x\in A;$$
\item The elements $h$ and $r$ belong to $M(B)$ with $h$ invertible and $r$ unitary and there exists a Jordan $^*$-isomorphism $S: A\to (B,\bullet_{r},*_{r})$ such that $h$ lies in $Z(B^{**},\bullet_{r},*_{r})= r Z(B^{**})$, and $$ T(x) = h r^* S(x) = S(x) r^* h, \hbox{ for all } x\in A;$$
\item $T$ is biorthogonality preserving;
\item $T$ is orthogonality preserving on $A_{sa}$;
\item $T$ is orthogonality preserving on $A^{+}$;
\item $T$ is biorthogonality preserving on $A_{sa}$;
\item $T$ is biorthogonality preserving on $A^+$;
\item $T$ preserves zero-triple-products, i.e. $$\{a,b,c\}=0 \Rightarrow \{T(a),T(b),T(c)\}=0;$$
\item $T$ preserves zero-triple-products in both directions, i.e. $$\{a,b,c\}=0 \Leftrightarrow \{T(a),T(b),T(c)\}=0.$$
\end{enumerate}
\end{corollary}

\section[One-parameter groups]{One-parameter groups of orthogonality preserving operators on a C$^*$-algebra}\label{sec: one-parameter semigroups Cstar algebra}

Let $X$ be a Banach space and let $B(X)$ stand for the Banach space of all bounded linear operators on $X$. A \emph{one-parameter semigroup} of bounded linear operators on $X$ is a correspondence $\mathbb{R}_0^{+} \to B(X),$ $t\mapsto T_t$ satisfying $T_{t+s} = T_{s} T_{t}$ for all $s,t\in \mathbb{R}_0^{+}$ and $T_0 =I$. It is a classic result that the mapping $t\mapsto T_t$ is uniformly continuous at the origin, i.e. $\displaystyle \lim_{t\to 0} \|T_t -I\| =0$, if and only if there exists a bounded linear operator $R\in B(X)$ such that $T_t =e^{t R}$ for all $t\in \mathbb{R}_0^+$, and in such a case, $T_t$ extends to a uniformly continuous one-parameter group on $\mathbb{R}$ (compare \cite[Proposition 3.1.1]{BratRob1987}).\smallskip

A surjective linear isometry between two C$^*$-algebras need not be, in general, a $^*$-isomorphism. R.V. Kadison settled the precise structure of all isometric linear bijections between unital C$^*$-algebras in \cite{Kad51}. The non-unital case is due to A.L.T. Paterson and A.M. Sinclair (see \cite{PatSinc72}). For each isometric linear bijection $T$ between C$^*$-algebras $A$ and $B$, there exists a unitary $u\in M(B)$ and a Jordan $^*$-isomorphism $\Phi: A\to B$ satisfying $T(a) = u \Phi (a)$ for all $a\in A$ (cf. \cite[Theorem 1]{PatSinc72}). In particular, any such a mapping $T$ satisfies $$\{T(x), T(y), T(z)\} = \frac12 ( u \Phi (x) \Phi(y)^* u^* u \Phi (z) + u \Phi (z) \Phi(y)^* u^* u \Phi (x) )$$ $$= \frac12 ( u \Phi (x) \Phi(y^*) \Phi (z) + u \Phi (z) \Phi(y^*) \Phi (x) ) = u \Phi \{x,y,z\} = T\{x,y,z\},$$ for all $x,y,z\in A$, that is, $T$ is a triple isomorphism. It is also known that every triple isomorphism between C$^*$-algebras is an isometry (see \cite[Proposition 3.4]{Harris81}). Therefore, for each C$^*$-algebra $A$ the set Iso$(A)$, of all surjective linear isometries on $A$, is precisely the set of all triple automorphisms on $A$. Clearly Iso$(A)$ is a subgroup of $B(A)$ (cf. \cite[Proposition 5.5]{Ka} for a more general statement). Another interesting subgroup of $B(A)$ is given by the set BiOP$(A),$ of all (bi-)orthogonality preserving bounded linear bijections on $A$ (let us observe that, by Corollary \ref{c Characterization bd OP plus bijective}, the prefix ``bi-'' can be inserted or relaxed without any problem). We can also consider the subgroup BiOP$_{sa}(A)$ of all elements in BiOP$(A)$ which are symmetric maps (i.e. $T(a)^* = T(a^*)$ for all $a\in A$).\smallskip

A complete description of all uniformly continuous semigroups of symmetric orthogonality preserving operators on unital C$^*$-algebras was obtained by M. Wolff as a consequence of his study on symmetric bounded linear operators  preserving orthogonality between unital C$^*$-algebras (see \cite[Theorem 2.6]{Wolff94}). Our aim here is to extend Wolff's study to uniformly continuous semigroups with values in BiOP$(A),$ where in our case $A$ will be a general C$^*$-algebra.\smallskip

In a first step towards describing one-parameter semigroups of orthogonality preserving operators it is convenient to begin with a result on uniformly continuous semigroups of surjective isometries on a C$^*$-algebra $A$. We recall that a derivation on $A$ is a linear mapping $D:A\to A$ satisfying $$D(a b) = D(a) b + a D(b),\hbox{ for all $a,b\in A$.}$$ If $A$ is unital, the binary Leibniz' rule in the line above implies that $D(1) =0$. Let us note that, by a result of S. Sakai, every derivation on a C$^*$-algebra is continuous \cite{Sak60}. A \emph{$^*$-derivation} on $A$ is a derivation which is also a symmetric mapping.\smallskip

For the next proposition we shall build our arguments on some results obtained by S. Pedersen in \cite{PedS88}.

\begin{proposition}\label{p uniparam semigroups o surjective isometries on unital C*-algebras} Let $\{U_t: t\in \mathbb{R}_0^+\}$ be a uniformly continuous one-parameter semigroup of surjective isometries on a C$^*$-algebra $A$. Suppose $U_t^{**} (1)=1$ for all $t\in \mathbb{R}_0^+$ {\rm(}i.e. $\{U_t: t\in \mathbb{R}_0^+\}$ is a uniformly continuous one-parameter semigroup of Jordan $^*$-automorphisms on $A${\rm)}. Then there exists a $^*$-derivation $D : A\to A$ satisfying $U_t = e^{t D}$ for all $t\in \mathbb{R}$, and $\{U_t: t\in \mathbb{R}_0^+\}$ is in fact a uniformly continuous one-parameter semigroup of {\rm(}associative{\rm)} $^*$-automorphisms on $A$.\smallskip

\noindent In the general case, there exist a $^*$-derivation $D : A^{**}\to A^{**}$ and an element $z_0\in A^{**}$ with $z_0^*= -z_0$ satisfying $U_t^{**} = e^{t (D+L_{z_0})}$ for all $t\in \mathbb{R}$. We can also find a $^*$-derivation $D_1 : A\to A$ and an element $z_1\in A^{**}$ with $z_1^*= -z_1$ satisfying $U_t^{**} = e^{t (D_1+M_{z_1})}$ for all $t\in \mathbb{R}$, where $M_{z_1} (x) = z_1\circ x$ {\rm(}$x\in A${\rm)}. When $A$ is unital we can take $^*$-derivations $D$ and $D_1$ on $A$ and skew symmetric elements $z_0,z_1\in A$.
\end{proposition}

\begin{proof} Suppose first that $U_t^{**} (1)=1$ for all $t\in \mathbb{R}$. Let us find $R\in B(A)$ such that $U_t = e^{t R}$ for all $t\in \mathbb{R}$ (cf. \cite[Proposition 3.1.1]{BratRob1987}). Since for each real $t$ the mapping $U_t$ is a Jordan $^*$-automorphism we have $$e^{t R} (a\circ b) = e^{t R}(a) \circ e^{t R}(b),\hbox{ and } e^{t R} (a)^* = e^{t R}(a^*) \ \ (\forall t\in \mathbb{R}, a,b\in A).$$ Taking derivatives at $t=0$ we get $$R (a\circ b) = R(a)\circ b + a\circ R(b), \hbox{ and } R(a)^* = R(a^*) \hbox{ for all } a,b\in A,$$ which proves that $R= D$ is a Jordan $^*$-derivation. It follows from \cite[Theorem 6.3]{John96} that $D$ is a $^*$-derivation.\smallskip

Suppose next that we are in the general case. Clearly, $\{U_t^{**}: t\in \mathbb{R}_0^+\}$ is a uniformly continuous (in particular, strongly continuous) one-parameter semigroup of surjective isometries on $A^{**}$. By \cite[Proposition 3.1.1]{BratRob1987} there exists $R\in B(A)$ such that $U^{**}_t = e^{t R^{**}}$ for all $t\in \mathbb{R}$.
%
By \cite[Corollary 3.3]{PedS88} the element $z_0 = R^{**}(1)\in A^{**}$ is skew symmetric (i.e. $z_0^* = -z_0$) and the mapping $D = R^{**} - L_{z_0}$ is a $^*$-derivation on $A^{**}$. For $z_1 = {z_0}$, the mapping $[\frac{z_0}{2}, \cdot]$ is a $^*$-derivation on $A^{**}$, and hence $D_1 = D + [\frac{z_0}{2},\cdot]$ is a $^*$-derivation with $R^{**} = D_1 + M_{z_1}$ and $z_1$ skew symmetric.
\end{proof}

\begin{remark}\label{r Sakai's thm on inner derivations} Another celebrated result due to Sakai proves that every (associative) derivation $D$ on a von Neumann algebra $M$ is inner, that is, there exists $z_0\in M$ satisfying $D(a) = [z_0,a]= z_0 a - a z_0$ for all $a\in M$ (see \cite[Theorem 4.1.6]{Sa}). In particular, every associative derivation annihilates on $Z(M)$, and every derivation on a commutative C$^*$-algebra is zero (see \cite[Lemma 4.1.2]{Sa}). If $a\in Z(A)$ and $D$ is a derivation on a C$^*$-algebra $A$ we have $$D L_a (x) = D(a x) = D(a) x + a D(x) = L_a D(x),$$ that is, $L_a$ and $D$ commute. Similarly, $D R_a = R_a D$ for $a$ and $D$ under these assumptions.\smallskip

We can therefore conclude from Proposition \ref{p uniparam semigroups o surjective isometries on unital C*-algebras} that every uniformly continuous one-parameter semigroup $\{U_t: t\in \mathbb{R}_0^+\}$ of Jordan $^*$-automorphisms on a commutative C$^*$-algebra $C$ satisfies $U_t(a) =a$ for all $t\in \mathbb{R},$ $a\in C$, that is, $\{U_t: t\in \mathbb{R}_0^+\}$ is constant $\{Id_{C}\}$.
\end{remark}

Let $a$ be an element in a C$^*$-algebra $A$ and let $a=u|a|$ be its polar decomposition. Let us denote by $C^*(|a|)$ and $C^*(|a|^3)$ the C$^*$-subalgebras of $A$ generated by $|a|$ and $|a|^3$, respectively. According to our notation, $a^{[3]} = u |a|^3$ is the polar decomposition of $a^{[3]}$. By applying the Stone-Weierstrass theorem it can be easily shown that $C^*(|a|)=C^*(|a|^3).$ We claim that the following property holds:
\begin{equation}\label{eq uniqueness of cubic root}\hbox{ $z^{[3]} = a$ for some $a,z$ in $A$ implies that $z=a^{[\frac13]}\in u C^*(|a|)$.}
\end{equation} Namely, let $z =w |z|$ be the polar decomposition of $z$. Clearly $a\in w C^*(|z|)$ and $a= w |z|^3$ is the polar decomposition of $a$, that is $u=w$ and $|a| = |z|^3$. Therefore $z= u |a|^{\frac13} \in u C^*(|a|)$.\smallskip

Our next goal is a generalization of Wolff's characterization of uniformly continuous  one-parameter semigroups of symmetric orthogonality preserving operators on unital C$^*$-algebras \cite[Theorem 2.6]{Wolff94}.

\begin{theorem}\label{t Wolff one-parameter for OP} Let $A$ be a C$^*$-algebra. Suppose $\{T_t: t\in \mathbb{R}_0^{+}\}$ is a family of orthogonality preserving bounded linear bijections on $A$ with $T_0=Id$. For each $t\geq 0$ let $h_t = T_t^{**} (1)$, let $r_t$ be the range partial isometry of $h_t$ in $A^{**}$ and let $S_t$ denote the triple automorphism on $A$ associated with $T_t$ given by Corollary \ref{c Characterization bd OP plus bijective}. The following are equivalent:\begin{enumerate}[$(a)$]\item $\{T_t: t\in \mathbb{R}_0^{+}\}$ is a uniformly continuous  one-parameter semigroup of orthogonality preserving operators on $A$;
\item $\{S_t: t\in \mathbb{R}_0^{+}\}$ is a uniformly continuous one-parameter semigroup of surjective linear isometries {\rm(}i.e. triple isomorphisms{\rm)} on $A$ {\rm(}and hence there exists a $^*$-derivation $D$ on $A^{**}$ and a skew symmetric element $z_1$ in $A^{**}$
    such that $S_t^{**} = e^{t (D+M_{z_1})}$ for all $t\in \mathbb{R}${\rm)}, the mapping $t\mapsto h_t $ is continuous at zero, and the identity \begin{equation}\label{eq new idenity in the statement of theorem 1 on one-parameter} h_{t+s} =  h_t r_t^* S_t^{**} (h_s)=T^{**}_t (h_s),
    \end{equation} holds for all $s,t\in \mathbb{R}.$
\end{enumerate}
\end{theorem}

\begin{proof} $(a)\Rightarrow (b)$ The semigroup $\{T_t: {t\in \mathbb{R}_0^+}\}$ extends to a uniformly continuous one-parameter group on $\mathbb{R}$. 
By hypothesis $T_t:A\to A$ is an orthogonality preserving and bijective bounded linear operator. By Corollary \ref{c Characterization bd OP plus bijective}, for each real $t$ there exist (unique) $h_t,r_t\in M(A)$ with $h_t$ invertible and $r_t$ unitary and a triple automorphism $S_t: A\to A$ satisfying $T^{**}_t (1) = h_t$, $r_t$ is the range partial isometry of $h_t$, $h_t^* S_t(x) = S_t(x^*)^* h_t$, $h_t S_t(x^*)^* = S_t(x) h_t^*$ (and hence $r_t^* S_t(x) = S_t(x^*)^* r_t$, $r_t S_t(x^*)^* = S_t(x) r_t^*$), $h_t r_t^* S_t(x) = S_t(x) r_t^* h_t$, and $T_t(x) = h_t r_t^* S_t(x) = S_t(x) r_t^* h_t,$ for all $x\in A$. We further know that $h_t, r_t\in Z(A^{**}, \bullet_{r_t},*_{r_t}) = r_t Z(A) = Z(A) r_t$ for all $t\in \mathbb{R}$.\smallskip

We therefore have two mappings $t\mapsto h_t$ and $t\mapsto S_t$ from $\mathbb{R}$ into $M(A)\cap (r_t Z(A^{**}))$ and Iso$(A)$, respectively. Clearly, $T_0 = Id$ implies that $h_0 = 1$.\smallskip

By applying that $\{T_t: t\in \mathbb{R}\}$ is a uniformly continuous group, the mapping $t\mapsto h_t = T_t^{**} (1) $ is continuous at zero, because $\| h_t -1\| = \|T_t^{**} (1) - T^{**}_0 (1)\|\leq \|T_t - T_0\|$, for all real $t$. It is well known that the mapping $a\mapsto a^{-1}$ is continuous on the set of invertible elements in $A^{**}$ (see, for example, \cite[Corollary I.1.8]{Tak}). We can therefore conclude that the mapping $t\mapsto h_t^{-1}$ is continuous at zero. Clearly, the mapping $t\mapsto |h_t|^2 = h_t^* h_t$ is continuous at zero with $|h_t|^2$ invertible for all $t$. It follows from \cite[Proposition I.4.10 and Corollary I.1.8]{Tak} that the mapping $t\mapsto |h_t|^{-1}$ is continuous at zero. Consequently, the mappings $t\mapsto r_t = h_t |h_t|^{-1}$ and $t\mapsto S_t = r_t^* h_t^{-1} T_t$ are continuous at zero.\smallskip

We claim that $\{ S_{t} : t\in \mathbb{R}\}$ is a one-parameter semigroup. It will be proved in several steps.\smallskip

First, a direct application of the hypotheses gives
\begin{equation}\label{eq hs+t} h_{t+s} = T_{t+s}^{**} (1) = T_{t}^{**} T_{s}^{**} (1) = h_t r_t^* S_t^{**} (h_s), \hbox{ for all } s,t\in \mathbb{R},
\end{equation} which proves \eqref{eq new idenity in the statement of theorem 1 on one-parameter}.\smallskip

Now, applying that 
$h_{t}, h_{t}^{[\frac{1}{3^n}]}\in Z(A^{**},\bullet_{r_{t}},*_{r_{t}})$ 
and that $S_t^{**}$ is a triple isomorphism, we can check that \begin{equation}\label{eq 3n cubic roots h} h_{t+s}^{[\frac{1}{3^n}]} = h_{t}^{[\frac{1}{3^n}]} r_t^* S_t^{**}(h_s^{[\frac{1}{3^n}]}), \hbox{ for all } n\in\mathbb{N}.
 \end{equation} Namely, arguing by induction on $n$, we assume that the identity holds for $n$ (the case $n=1$ essentially follows by the same arguments).  Since $h_{t}, h_{t}^{[\frac{1}{3^n}]}\in Z(A^{**},\bullet_{r_{t}},*_{r_{t}}),$ the triple product also satisfies that $$\{x,y,z\} = \frac12 (x \bullet_{r_t} y^{*_{r_t}} \bullet_{r_t} z + z \bullet_{r_t} y^{*_{r_t}} \bullet_{r_t} x), \hbox{ for all } x,y,z\in A,$$ and $S_t^{**}$ is a triple isomorphism (and thus $S_t : A\to (A,\bullet_{r_t}, *_{r_t})$ is a Jordan $^*$-isomorphism) we deduce that $$ \left(h_{t}^{[\frac{1}{3^{n+1}}]} r_t^* S_t^{**}(h_s^{[\frac{1}{3^{n+1}}]})\right)^{[3]} = \left(h_{t}^{[\frac{1}{3^{n+1}}]}\right)^{[3]} \left(r_t^*\right)^{[3]} S_t^{**}\left(h_s^{[\frac{1}{3^{n+1}}]}\right)^{[3]}   $$
$$ = h_{t}^{[\frac{1}{3^{n}}]} r_t^* S_t^{**}\left( \left(h_s^{[\frac{1}{3^{n+1}}]}\right)^{[3]}\right) = h_{t}^{[\frac{1}{3^{n}}]} r_t^* S_t^{**}\left(h_s^{[\frac{1}{3^{n}}]} \right) = h_{t+s}^{[\frac{1}{3^n}]},$$ by the induction hypothesis. It follows from \eqref{eq uniqueness of cubic root} that $h_{t}^{[\frac{1}{3^{n+1}}]} r_t^* S_t^{**}(h_s^{[\frac{1}{3^{n+1}}]}) =  h_{t+s}^{[\frac{1}{3^{n+1}}]},$ which concludes the induction argument.\smallskip

Now, since $(h_{t+s}^{[\frac{1}{3^n}]})_n\to r_{t+s}$, $(h_{t}^{[\frac{1}{3^n}]})_n\to r_{t}$ and $(h_{s}^{[\frac{1}{3^n}]})_n\to r_{s}$ in the strong$^*$ topology and the product of every von Neumann algebra is jointly strong$^*$ continuous on bounded sets (cf. \cite[Proposition 1.8.12]{Sa}), by taking strong$^*$ limits in \eqref{eq 3n cubic roots h} we derive \begin{equation}\label{eq r t+s via St} r_{t+s} = r_t r_t^* S_t^{**} (r_s) = S_t^{**}(r_s),
 \end{equation} where we have applied that $S_t^{**}$ is strong$^*$ continuous.
Furthermore, since  $S_t : A\to (A,\bullet_{r_t}, *_{r_t})$ is a Jordan $^*$-isomorphism \begin{equation}\label{eq rt+s} r_{t} r_{t+s}^* r_{t}= r_t  S_t^{**}(r_s)^* r_{t} = S_t^{**}(r_s)^{*_{r_t}}= S_t^{**} (r_s^*).
\end{equation}

Having in mind that $S_t: A\to A$ is a triple isomorphism, it follows from \eqref{eq r t+s via St} that
\begin{equation}\label{eq St is a Jordan star isom from rt to rt+s} \hbox{$S_t: (A, \bullet_{r_s}, *_{r_{s}}) \to (A, \bullet_{r_{t+s}}, *_{r_{t+s}})$ is a Jordan $^*$-isomorphism.}
\end{equation}

Fix now $a\in A$ and consider the identity
\begin{align} h_t r_t^* S_t^{**} (h_s) r_{t+s}^* S_{t+s} (a) &= \hbox{(by \eqref{eq hs+t})}=  h_{t+s} r_{t+s}^* S_{t+s} (a)  = T_{t+s} (a)   \nonumber \\
&= T_{t} T_{s} (a) = h_t r_t^* S_t (h_s r_s^* S_s(a)).
\end{align}

Having in mind that $ h_t r_t^*$ is invertible we get
\begin{equation}\label{eq Feb02 1} S_t^{**} (h_s) r_{t+s}^* S_{t+s} (a) = S_t (h_s r_s^* S_s(a)).
\end{equation}
Since the element $h_s r_s^*\in Z(A) r_s r_s^* = Z(A)$ and $S_t : A\to (A,\bullet_{r_t}, *_{r_t})$ is a Jordan $^*$-isomorphism, the comments in Remark \ref{r invertibility and center in the homotope}~\eqref{eq element in the center is like a homom} imply that the element in the right-hand-side of the previous equality writes in the form \begin{equation}\label{eq a 09 02 2020} S_t (h_s r_s^* S_s(a)) = S_t^{**} (h_s r_s^*) r_t^* S_t (S_s(a)).
\end{equation}

Now, by applying that $S_t: (A, \bullet_{r_s}, *_{r_{s}}) \to (A, \bullet_{r_{t+s}}, *_{r_{t+s}})$ is a Jordan $^*$-isomorphism and $h_s\in Z (A, \bullet_{r_s}, *_{r_{s}})$, we deduce from Remark \ref{r invertibility and center in the homotope}~\eqref{eq element in the center is like a homom} that $$ S_t^{**} (h_s r_s^*)  = S_t^{**} (h_s \bullet_{r_s} 1) = S_t^{**} (h_s) \bullet_{r_{t+s}} S_t^{**}(1) = S_t^{**} (h_s) r_{t+s}^* r_t,$$ which combined with \eqref{eq a 09 02 2020} gives
$$S_t (h_s r_s^* S_s(a)) = S_t^{**} (h_s) r_{t+s}^* r_t r_t^* S_t (S_s(a))= S_t^{**} (h_s) r_{t+s}^* S_t S_s(a).$$ By combining this identity with \eqref{eq Feb02 1} and having in mind that $S_t^{**} (h_s) r_{t+s}^*$ must be invertible, we arrive at  $S_{t+s} (a) = S_t S_s(a),$ for all $a\in A$, $t,s\in \mathbb{R}$. \smallskip

We have shown that $\{S_t : t\in \mathbb{R}\}$ is a group which is continuous at zero, equivalently, $\{ S_{t} : t\in \mathbb{R}\}$ is a uniformly continuous one-parameter group of surjective isometries on $A$. By Proposition \ref{p uniparam semigroups o surjective isometries on unital C*-algebras} there exists a (continuous) $^*$-derivation $D : A^{**}\to A^{**}$ and a skew symmetric element $z_1\in A^{**}$ such that $S_t^{**} = e^{t(D+M_{z_1})}$ for all $t\in \mathbb{R}$.\smallskip

$(b)\Rightarrow (a)$ By hypotheses, $\{S_t: t\in \mathbb{R}_0^{+}\}$ is a uniformly continuous one-parameter semigroup of surjective linear isometries (i.e. triple isomorphisms) on $A,$ the mapping $t\mapsto h_t $ is continuous at zero, and the identity \eqref{eq new idenity in the statement of theorem 1 on one-parameter} holds for all $t\in \mathbb{R}$. Since $h_t$ is invertible for all $t$ in $\mathbb{R},$ the arguments in the proof of $(a)\Rightarrow (b)$ are valid to prove that the mappings $t\mapsto r_t = h_t |h_t|^{-1}$ and $t\mapsto T_t = h_t r_t^* S_t$ are continuous at zero.\smallskip

We shall finally prove that $\{T_t: t\in \mathbb{R}_0^{+}\}$ is a one-parameter semigroup. As in the proof of $(a)\Rightarrow (b)$~\eqref{eq r t+s via St}, we can deduce from \eqref{eq new idenity in the statement of theorem 1 on one-parameter} that $r_{t+s} = r_t r_t^* S^{**} (r_s) = S_t(r_s)$ for all $s,t\in \mathbb{R}^+_0$. Therefore $$S_t: (A, \bullet_{r_s}, *_{r_{s}}) \to (A, \bullet_{r_{t+s}}, *_{r_{t+s}})$$ is a Jordan $^*$-isomorphism. Since $h_s\in Z(A^{**},\bullet_{r_s},*_{r_{s}})$, it follows from Remark \ref{r invertibility and center in the homotope}\eqref{eq element in the center is like a homom} that
$$\begin{aligned}T_t T_s (a) &= h_t r_t^* S_t (h_t r_s^* S_s(a))= h_t r_t^* S_t (h_t \bullet_{r_s} S_s(a)) \\ &= h_t r_t^* S_t^{**} (h_t) \bullet_{r_{t+s}} S_t (S_s(a))= h_{t+s} {r_{t+s}^*} S_t (S_s(a)) \\
&= h_{t+s} {r_{t+s}^*} S_{t+s} (a) =T_{t+s} (a),
 \end{aligned}$$ for all $a\in A$, where at the antepenultimate and at the penultimate equalities we applied \eqref{eq new idenity in the statement of theorem 1 on one-parameter} and the fact that $\{S_t: t\in \mathbb{R}_0^{+}\}$ is a one-parameter semigroup, respectively.
\end{proof}

Suppose that in Theorem \ref{t Wolff one-parameter for OP} we additionally assume that each $T_t$ is a symmetric operator, that is,  $\{T_t: t\in \mathbb{R}_0^{+}\}$ is a uniformly continuous one-parameter semigroup of orthogonality preserving symmetric operators on $A$. By Corollary \ref{c Characterization bd OP plus bijective}, the elements $r_t$ and $h_t$ lie in $M(A)_{sa}$ with $h_t$ invertible and $r_t$ a symmetric unitary, or equivalently, a symmetry. 
It is explicitly shown in \cite[$(15)$ and $(16)$ in the proof of Theorem 17]{BurFerGarMarPe2008} that $r_t T_t(a) = T_t(a^*)^* r_t$ and $h_t T_t(a) = T_t(a^*)^* h_t$ for all $a\in A$ $t\in \mathbb{R}$. Having in mind that $T_{t}$ is symmetric and bijective we conclude that $h_t$ and $r_t$ both lie in $Z(A^{**})$.\smallskip

On the other hand, we know from Theorem \ref{t Characterization bd OP plus surjective} above that $\{S_t\}_{t\in \mathbb{R}}$ is a one-parameter group. We claim that, under the extra assumptions here, $\{r_t S_t\}_{t\in \mathbb{R}}$ is a one-parameter group of Jordan $^*$-isomorphisms. Namely, we have seen in the proof of the just quoted theorem that $S_t : A\to (A,\bullet_{r_t}, *_{r_t})$ is a Jordan $^*$-isomorphism. Since $r_s\in Z(A^{**})$, we deduce from Remark \ref{r invertibility and center in the homotope}\eqref{eq element in the center is like a homom} that $$r_t S_t (r_s S_s (a)) = r_t S_t^{**} (r_s) r_t^* S_t S_s (a) = r_{t+s} S_{t+s} (a), \hbox{ for all } a\in A,$$
where in the last equality we applied \eqref{eq r t+s via St} from the proof of Theorem \ref{t Characterization bd OP plus surjective} and the fact that $r_t$ is central. This finishes the proof of the claim.\smallskip

Combining that $\{r_t S^{**}_t|_{Z(M(A))} : Z(M(A))\to Z(M(A)) \}_{t\in \mathbb{R}}$ is a one-parameter group of Jordan $^*$-isomorphisms, the fact that $h_s\in Z(A^{**})$ and Remark \ref{r Sakai's thm on inner derivations} we get $$h_{t+s} =T_{t+s} (1)= T_t T_s (1) = h_t (r_t S_t) (h_s) = h_t h_s,$$ and clearly $r_{t+s} =  r_t r_s,$  for all $t,s\in \mathbb{R}.$\smallskip

Finally, since the mapping $t\mapsto L_{h_t}|_{Z(M(A))}$ is a uniformly continuous one-parameter group on the unital commutative C$^*$-algebra $Z(M(A))$, we can find a bounded linear operator $R\in B(Z(M(A)))$ satisfying $L_{h_t}|_{Z(M(A))} = e^{t R}$ for all $t\in \mathbb{R}$ (cf. \cite[Proposition 3.1.1]{BratRob1987}). This implies the existence of $h\in Z(M(A))$ such that $h_t = e^{t h}$ for all $t\in \mathbb{R}$. Now an application of Proposition \ref{p uniparam semigroups o surjective isometries on unital C*-algebras} proves the existence of a $^*$-derivation $d$ on $A$ satisfying $$T_t (a) = e^{t h} (r_t S_t) (a) = e^{t h} e^{t d} (a) = e^{t (L_h+ d)} (a),$$ for all $a\in A$, $t\in \mathbb{R}$ (here we applied that $h$ is central). Clearly for $h$ and $d$ as above the set $\{ e^{t (L_h+ d)} : t\in \mathbb{R}\}$ is a uniformly continuous one-parameter group.\smallskip

We have found a non-unital version of the result proved by M. Wolff in \cite[Theorem 2.6]{Wolff94}.

\begin{corollary}\label{c Wolff one-parameter for OP symmetric} Let $A$ be a C$^*$-algebra. Suppose $\{T_t: t\in \mathbb{R}_0^{+}\}$ is a family of orthogonality preserving and symmetric bounded linear bijections on $A$ with $T_0=Id$. For each $t,$ let $h_t = T_t^{**} (1)$, let $r_t$ be the range partial isometry of $h_t$ in $A^{**}$ and let $S_t$ denote the triple isomorphism associated with $T_t$ given by Corollary \ref{c Characterization bd OP plus bijective}. The following are equivalent:\begin{enumerate}[$(a)$]\item $\{T_t: t\in \mathbb{R}_0^{+}\}$ is a uniformly continuous  one-parameter semigroup of orthogonality preserving operators on $A$;
\item There exists $h\in Z(M(A))$ and a $^*$-derivation $d$ on $A$ such that $h_t = e^{t h}$ and $$T_t (a) =  e^{t (L_h+ d)} (a),$$ for all $a\in A$, $t\in \mathbb{R}$.
\end{enumerate}

\noindent Moreover, if any of the previous statements holds, the sets $\{r_t:t\in \mathbb{R}\}$ and  $\{h_t:t\in \mathbb{R}\}$ are one-parameter groups in $Z(M(A))$.
\end{corollary}\smallskip

\begin{remark}\label{R the ht are not a group in the general case} Let us observe that under the more general hypotheses of Theorem \ref{t Wolff one-parameter for OP} we cannot, in general, conclude that the sets $\{r_t:t\in \mathbb{R}\}$ and  $\{h_t:t\in \mathbb{R}\}$ are one-parameter groups. Let $a$ be an element in a C$^*$-algebra $A$. Let us consider the mapping $i L(a,a) : A\to A,$ $x\mapsto i \{a,a,x\} = \frac{i}{2} (a a^* x + x a^*a)$. The elements $a a^*,$ $a^* a$ are positive and in particular hermitian elements (i.e. they have real numerical range in the usual sense employed in \cite{BonsDun73}). An standard argument, applying the representation from $A$ into $B(A)$ via the left and the right multiplication operators given by $a\mapsto L_a$ and $a\mapsto R_a$, proves that $L_{a a^*}$ and $R_{a^* a}$ are hermitian elements in $B(A)$ and the same applies to $\frac12 (L_{a a^*}+R_{a^*a})$. We can therefore conclude that $e^{t i L(a,a) } = e^{\frac{i t }{2} (L_{a a^*}+R_{a^*a})}$ is a surjective real linear isometry for each real $t$ (cf. \cite[Corollary 10.13]{BonsDun73}). That is the set $\{ e^{t i L(a,a) }  : t\in \mathbb{R}\}$ is a one-parameter group of surjective linear isometries on $A$.\smallskip

If $a=e$ is a partial isometry in $A$ the mapping $i L(e,e)$ is precisely $i P_2(e) +\frac{i}{2} P_1 (e)$. Thus $$ T_t =e^{i t L(e,e)} = e^{it} P_2(e) +e^{\frac{i t}{2}} P_1 (e) +  P_0 (e), \ \forall t\in \mathbb{R}.$$ Take, for example, $A=M_2 (\mathbb{C})$ and a partial isometry $e\in A$ such that $e e^* = \frac12 \left(
                   \begin{array}{cc}
                     1 & 1 \\
                     1 & 1 \\
                   \end{array}
                 \right)
$ and $e^* e = \left(
                   \begin{array}{cc}
                     1 & 0 \\
                     0 & 0 \\
                   \end{array}
                 \right)
$. It is easy to check that $$\begin{aligned}  e^{i t L(e,e)} \left(
                                                \begin{array}{cc}
                                                  \alpha_{11} & \alpha_{12} \\
                                                  \alpha_{21} & \alpha_{22} \\
                                                \end{array}
                                              \right) & = e^{it}\frac{\alpha_{11}+\alpha_{21}}{2}  \left( \begin{array}{cc}
                                                   1 &  0 \\ 1 & 0 \\
                                                \end{array}
                                              \right)
                                              + e^{\frac{it}{2}}
                                              \frac{\alpha_{11}-\alpha_{21}}{2} \left( \begin{array}{cc}
                                                  1 &  0 \\ -1 & 0\\
                                                \end{array}
                                              \right) \\
&
 +e^{\frac{it}{2}}
                                              \frac{\alpha_{12}+\alpha_{22}}{2} \left( \begin{array}{cc}
                                                  0 &  1 \\ 0 & 1 \\
                                                \end{array}
                                              \right)
 +
                                              \frac{\alpha_{12}-\alpha_{22}}{2} \left( \begin{array}{cc}
                                                  0 &  1 \\ 0 & -1 \\
                                                \end{array}
                                              \right)
\end{aligned},$$ and thus $h_t=r_t = e^{i t L(e,e)} \left(
                                                \begin{array}{cc}
                                                  1 & 0 \\
                                                  0 & 1 \\
                                                \end{array}
                                              \right) = \left(
                                                \begin{array}{cc}
                                                  \frac{e^{\frac{it}{2}}(e^{\frac{it}{2}}+1)}{2} & \frac{e^{\frac{it}{2}}-1}{2} \\
                                                  \frac{e^{\frac{it}{2}}(e^{\frac{it}{2}}-1)}{2} & \frac{e^{\frac{it}{2}}+1}{2} \\
                                                \end{array}
                                              \right).$ 
It is a bit laborious to check that $$r_t r_s -r_{t+s} = \left(
                                                          \begin{array}{cc}
                                                            \frac{-1}{4} e^{\frac{i s}{2}} (-1+e^{\frac{i s}{2}}) (-1+e^{\frac{i t}{2}})^2 & \frac{1}{4} (-1+e^{\frac{i s}{2}}) (-1+e^{{i t}}) \\
                                                            \frac{-1}{4} e^{\frac{i s}{2}} (-1+e^{\frac{i s}{2}}) (-1+e^{{i t}}) & \frac{1}{4}  (-1+e^{\frac{i s}{2}}) (-1+e^{\frac{i t}{2}})^2 \\
                                                          \end{array}
                                                        \right)\neq 0,$$ for $s,t\in \mathbb{R}\backslash\{0\}$.
We can modify the previous example with $\widetilde{T}_t = e^t e^{t \delta}$ to obtain a more general $h_t = e^t r_t$ with a similar bad behaviour. These conclusions are somehow related to the study developed by S. Pedersen in \cite[Theorem 3.8]{PedS88}, where from a slightly different argument and statemet it is shown that if $\{ S_t = e^{t \delta} : t\in\mathbb{R}\}$ is a uniformly continuous one-parameter semigroup of surjective isometries (i.e. triple automorphisms) on a unital C$^*$-algebra $A,$ where $\delta$ is a triple derivation on $A$, the following are equivalent:\begin{enumerate}[$(1)$]\item $r_s r_t = r_{t+s}$, for all $t,s\in \mathbb{R};$
\item $r_s = (r_t^* S_t) (r_s)$ for all $t,s\in \mathbb{R};$
\item $L_{r_t} (r_t^* S_t) = (r_t^* S_t) L_{r_t}$ for all $t\in \mathbb{R}$;
\item $\delta^2 (1) = \delta(1)^2$,
\end{enumerate} where $r_t=S_t (1)$.\smallskip

By repeating the above arguments with $v= \left(
                                            \begin{array}{cc}
                                              0 & 1 \\
                                              0 & 0 \\
                                            \end{array}
                                          \right),$ and $S_t = e^{i t L(v,v)}$ we can easily check that $r_t =  e^{i t L(v,v)} \left(
                                                                                                                                 \begin{array}{cc}
                                                                                                                                   1 & 0 \\
                                                                                                                                   0 & 1 \\
                                                                                                                                 \end{array}
                                                                                                                               \right)
                                           = e^{ \frac{i t}{2}} \left(
                                                                                                                                 \begin{array}{cc}
                                                                                                                                   1 & 0 \\
                                                                                                                                   0 & 1 \\
                                                                                                                                 \end{array}
                                                                                                                               \right),$
and thus $r_{t+s} = r_t r_s$ for all $s,t\in \mathbb{R}$. The last identity can be directly checked or deduced from the previous equivalences because for $\delta = i L(v,v) $ we have $\delta^2 (1) = \delta(1)^2$. We note that $T_t$ is not symmetric for any non-zero $t$. That is, the final conclusion in Corollary \ref{c Wolff one-parameter for OP symmetric} may hold for one-parameter semigroups of non-necessarily symmetric surjective isometries.\smallskip

\end{remark}

\textbf{Acknowledgements} The authors are grateful to the referee for careful reading of the paper and valuable suggestions and comments.\smallskip

A.M. Peralta partially supported by the Spanish Ministry of Science, Innovation and Universities (MICINN) and European Regional Development Fund project no. PGC2018-093332-B-I00, Junta de Andaluc\'{\i}a grant FQM375 and Proyecto de I+D+i del Programa Operativo FEDER Andalucia 2014-2020, ref. A-FQM-242-UGR18. \smallskip

%
%



\medskip\medskip

\end{document}